\documentclass[11pt,letterpaper]{article}
\usepackage[margin=1.5in]{geometry}
\usepackage{amsmath,times}
\usepackage{amssymb,amsmath}
\usepackage{graphicx}
\usepackage{amsthm}
\usepackage{multicol}
\usepackage[all]{xy}
\usepackage{cancel}
\usepackage{tikz}
\usepackage[utf8]{inputenc}
\usepackage{url}

\title{Special cases on the relative rank of G-equivariant functions over an infinite G-set}
\author{Ram\'on H. Ruiz-Medina\footnote{Email: harath.ruiz@academicos.udg.mx}  \\[1em]
\small{Centro Universitario de Ciencias Exactas e Ingenier\'ias}, \\ 
\small{Universidad de Guadalajara, Guadalajara, M\'exico.}}
\date{}

\newtheorem{teorema}{Theorem}[]
\newtheorem{lema}[teorema]{Lemma}

\newtheorem{corolario}[teorema]{Corollary}

\newcommand{\EndG}{\mathrm{End}_{G}(X)}
\newcommand{\AutG}{\mathrm{Aut}_{G}(X)}
\newcommand{\Hbox}{\mathcal{B}_{[H]}}
\newcommand{\Gbox}{\mathcal{B}_{[G]}}
\newcommand{\StabG}{\mathrm{Stab}_{G}(X)}
\newcommand{\ConjG}{\mathrm{Conj}_{G}}
\newcommand{\ConjX}{\mathrm{Conj}_{G}(X)}

\newcommand{\EndH}{\mathrm{End}_{G}(\Hbox)}

\newcommand{\EndHi}{\mathrm{End}_{G}(\ibox)}
\newcommand{\AutHi}{\mathrm{Aut}_{G}(\ibox)}
\newcommand{\EndHr}{\mathrm{End}_{G}(\rbox)}
\newcommand{\AutHr}{\mathrm{Aut}_{G}(\rbox)}

\newcommand{\EndGx}{\mathrm{End}_{G}(Gx)}
\newcommand{\AutGx}{\mathrm{Aut}_{G}(Gx)}
\newcommand{\EndGbox}{\mathrm{End}_{G}(\Gbox)}
\newcommand{\AutGbox}{\mathrm{Aut}_{G}(\Gbox)}
\newcommand{\ibox}{\mathcal{B}_{i}}

\newcommand{\rbox}{\mathcal{B}_{r}}

\newcommand{\unobox}{\mathcal{B}_{1}}

\begin{document}

\maketitle

\begin{abstract}
Given the action of a group $G$  on a set $ X $ , the set of $ G $ -equivariant functions, those that commute with the action, i.e., $ f(g \cdot x) = g \cdot f(x) $  for all $ x \in X $ , $ g \in G $ , forms a monoid under function composition. It is known that when a finite group $ G $  acts on a finite set $ X $ , a finite number of elements in this monoid are sufficient to generate it along with its group of units. The minimum number of elements required to generate the monoid is called the relative rank.
This paper presents two particular cases where $ G $  is a finite group, $ X $  is an infinite set, and the relative rank of the monoid of $ G $ -equivariant functions is finite; moreover, we compute it.\\
\textbf{Keywords:} $ G $ -set, $ G $ -equivariant function, group of units, relative rank. \\

\textbf{MSC 2020:} 20B25, 20E22, 20M20.
\end{abstract}

\section{Introduction}

In semigroup theory, rank and relative rank are widely studied objects. Given a monoid $ M $  and its group of units $ U $ , we define the relative rank of $ M $  modulo $ U $  as the minimum cardinality of a set $ W \subseteq M $  such that $ \langle U \cup W \rangle = M $ . The main goal of this work is to determine the relative rank of a certain monoid given its group of units.\\

Given a group $ G $  acting on a set $ X $ , a $ G $ -equivariant function is a function $ f: X \rightarrow X $  such that $ f(g \cdot x) = g \cdot f(x) $  for all $ g \in G $ , $ x \in X $ . The set of all $ G $ -equivariant functions forms a monoid under function composition. We call $ X $  a $ G $ -set and endomorphisms of $ X $  all $ G $ -equivariant functions. We denote by $ \EndG $  the monoid of endomorphisms of $ X $ , and by $ \AutG $  its group of units, which consists of all bijective $ G $ -equivariant functions.

We work with a finite group $ G $  acting on a set $ X $ , and construct a subset in $ \EndG $  such that this subset generates $ \EndG $  modulo $ \AutG $ . We assert that this subset has the minimal cardinality of any such set with this property. We construct this subset based on a subset of $ X $  and the restriction of endomorphisms to this subset.\\

Given the action of a group $ G $  on a set $ X $ , we recall the $ G $ -orbits and the stabilizer of elements in $ X $  as follows, for $ x \in X $ :
$$  Gx := \{ g \cdot x \mid g \in G \}, $$ 
$$  G_{x} := \{ g \in G \mid g \cdot x = x \}, $$ 
and based on the stabilizer, we define the following sets in $ X $ . Given $ H \leq G $ , let:
$$  \mathcal{B}_{H} := \{ x \in X \mid G_{x} = H \}, $$ 
$$  \Hbox := \{ x \in X \mid [G_{x}] = [H] \}. $$

\noindent Denote $ X/G $  and $ \Hbox/ G $  as the orbits of $ G $  in $ X $  and $ \Hbox $  respectively.\\

For a finite group $ G $ , denote $ [H_{1}], [H_{2}],...,[H_{r}] $  as all conjugacy classes of subgroups of $ G $ , ordered by their cardinality as:
$$  |H_{1}| \leq |H_{2}| \leq ... \leq |H_{r}|. $$ 

In the case where $ G $  is a finite group, we simplify the notation of sets $ \mathcal{B}_{[H_{i}]} $  as $ \mathcal{B}_{[H_{i}]}=\ibox $ . Note that some conjugacy classes may not be included in the stabilizer set of the action of $ G $  on $ X $ ,
$$  \StabG:=\{G_{x}|\ x\in X\}. $$

However, we must note that if $ H \in \StabG $ , the complete conjugacy class of $ H $ , $ [H] $ , is contained in $ \StabG $ , because if $ h \in G_{x} $  for some $ x \in X $ , it holds that 
$$  (g^{-1}hg) \cdot (g^{-1}\cdot x) = g^{-1}\cdot x, \ \forall g \in G, $$ 
meaning any conjugate of $ h $  stabilizes at least one element in $ X $ .\\

The following result is well known in the theory of $ G $ -equivariant functions. 
\begin{lema}\label{lema1}
Let $ G $  be a group acting on a set $ X $ , given $ x,y \in X $ , the following holds: 

\begin{enumerate} 
\item[i)] There exists a $ G $ -equivariant function $ \tau \in \EndG $  such that $ \tau(x)=y $  if and only if $ G_{x} \leq G_{y} $ .
\item[ii)] There exists a bijective $ G $ -equivariant function $ \sigma \in \AutG $  such that $ \sigma(x)=y $  if and only if $ G_{x} = G_{y} $ .
\end{enumerate}
\end{lema}

Further details on this result can be found in \cite{paper}.\\

Denote the set of all conjugacy classes of subgroups of $ G $  as $ \ConjG $ , and denote them as follows: $ [H]:=\{g^{-1}Hg:\ g\in G\} $ . Given subgroups $ H $  and $ N $  of $ G $ , we define the $ N $ -conjugacy classes of $ H $  as:
$$  [H]_{N}:=\{n^{-1}Hn:\ n\in N\}. $$ 
It is easy to see that $ [H]_{N} \subseteq [H] $ , meaning elements in an $ N $ -conjugacy class of $ H $  are some of the conjugate subgroups of $ H $ , specifically those given by conjugating elements in $ N $ . Denote the normalizer of a subgroup $ H $  as $ N_{G}(H)=N_{H} $ .\\

Let's recall some important definitions for properties of functions that are useful for our objectives. Let $ f:X\rightarrow X $  be a function, the range of $ f $  is the cardinality of its image, $ r(f)=|Im(f)| $ . It is known that the following property holds:
$$  r(f\circ g)\leq \min\{r(f),r(g)\}. $$ 
The defect of $ f $  is defined as the cardinality of the complement of the image of $ f $ , $ def(f)=|X\setminus Im(f)|=|X|- r(f) $ . The kernel of $ f $  is given by:
$$  ker(f):=\{(x,y)\in X\times X:\ f(x)=f(y)\}. $$ 
The kernel of $ f $  defines an equivalence relation on $ X $ . If we consider the equivalence class of an element $ x\in X $ , $ [x]_{f}:=\{y\in X:\ f(y)=x\} $ , the infinite contraction index of $ f $  is defined as the number of classes in the kernel of size $ |X| $ , i.e.,
$$  k(f):=|\{[x]_{f}\in X/ker(f):\ |[x]_{f}|=|X|\}|. $$ 
Note that if $ X $  is a finite set, $ \kappa(f)\neq 0 $  if and only if $ f $  is a constant function, and $ \kappa(f)=1 $ . This quantity is interesting and useful only when $ X $  is an infinite set. \\

Given subgroups $ H,K\leq G $ , such that $ H\leq K $ , we call a function $ \tau \in \EndG $  an elementary collapse of type $ (H,[K]_{N_{[H]}}) $  if for any $ x\in X $  such that $ G_{x}=H $ , the following conditions hold:

\begin{enumerate}
\item[1.]  $ [G_{\tau(x)}]_{N_{H}}=[K]_{N_{H}} $ .
\item[2.] $ ker(\tau)=\{(g\cdot x,g\cdot \tau(x)),(g\cdot \tau(x), g\cdot x):\ g\in G\}\cup\{(a,a):\ a\in X\} $ 
\end{enumerate}

The following results are presented in \cite{paper}.

\begin{lema}
    Let $ G $  be a group acting on a set $ X $ , the number of distinct types of elementary collapses is given by:
    $$  \sum_{[H]\in \ConjG}{|U(H)|}-|\kappa_{G}(X)| $$ 
    
where $ \kappa_{G}(X) $  is the set of conjugacy classes of subgroups of $ G $  such that $ \Hbox $  contains only one orbit of $ G $ , and $ U(H):=\{[K]_{N_{H}}|\ H\leq K\} $  is the set of elementary collapses of $ H $  into larger subgroups included in the action of $ G $  on $ X $ .
\end{lema}

\begin{teorema}
Let $ G $  be a finite group acting on a finite set $ X $ , then any subset $ U\subseteq \EndG $  that generates $ \EndG $  modulo $ \AutG $  must contain an elementary collapse of each possible type.
\end{teorema}

An extension of this result in \cite{mini1}, when $ X $  may not be finite, provides an upper bound for the relative rank of $ \EndG $  modulo $ \AutG $  as follows:

$$  \sum_{[H]\in \ConjG}{|U(H)|}-|\kappa_{G}(X)| \leq \mathrm{rank}(End_{G}(X):Aut_{G}(X)). $$

In the article \cite{paper}, we find results like Theorem \cite[Theorem 17]{paper} and its proof, as well as Lemma \cite[Lemma 11]{paper}, which provide us with tools to construct generating sets for the monoid when $ X $  is a finite set and $ G $  is a finite group. However, the generating set proposed in Theorem \cite[Theorem 17]{paper} may not generate the entire $ End_{G}(X) $  modulo $ Aut_{G}(X) $  in the case where $ X $  is an infinite set. We propose a new generating set based on the following constructions. We use constructions analogous to those found in the main theorems of \cite{Ara} and \cite{cite3} for the monoid of complete transformations of a set to achieve our goals in the monoid of $ G $ -equivariant functions of a $ G $ -set $ X $ .

\section{Restriction and Induction of $G$-equivariant Functions}

Consider a finite group $ G $  acting on a set $ X $ . We propose some $ G $ -equivariant functions defined from a set of representatives of the $ G $ -orbits within a set $ \Hbox $ . 

Let $ H $  be a subgroup of $ G $  and let $ \Lambda $  be an index set with the same cardinality as the set of $ G $ -orbits $ \Hbox /G $ . We can fix elements $ x_{\lambda}\in \mathcal{B}_{[H]} $ , $ \lambda \in \Lambda $ , without loss of generality, such that $ Gx_{\lambda}\neq Gx_{\kappa} $  and $ G_{x_{\lambda}}=G_{x_{\kappa}}=H $ , for $ \lambda,\kappa\in \Lambda $ .

Given a fixed subgroup $ H_{i} $ , denote $ X_{i} $  as the set of representatives of the $ G $ -orbits in $ \ibox $ . This construction allows us to state that $ |\ibox|=|X_{i}\times G/H_{i}|=|X_{i}|\cdot [G:H_{i}] $ . If $ |\ibox|=\infty $ , then $ |X_{i}|=\infty $ .

Note that for any element $ z\in \ibox $ , there exist $ g\in G $ , $ \lambda \in \Lambda $ , and $ x_{\lambda}\in X_{i} $  such that $ z=g\cdot x_{\lambda} $ . Given a function $ f:X_{i} \rightarrow X_{i} $ , we define a function $ \widetilde{f}:\ibox \rightarrow \ibox $  as:
$$  \widetilde{f}(z)=\widetilde{f}(g\cdot x_{\lambda})=g\cdot f(x_{\lambda}). $$ 

The function $ \widetilde{f} $  is well-defined because if $ z\in \ibox $  such that $ z=g\cdot x=h\cdot x $ , then there exists $ \lambda \in \Lambda $  and $ g'\in G $  such that $ x=g'\cdot x_{\lambda} $ , implying $ gg'\cdot x_{\lambda}=hg'\cdot x_{\lambda} $ . Without loss of generality, we rewrite this expression as $ g\cdot x_{\lambda}=h\cdot x_{\lambda} $ . Hence, $ h^{-1}g\in G_{x_{\lambda}} $ , and therefore
$$  \widetilde{f}(x_{\lambda})=\widetilde{f}(h^{-1}g\cdot x_{\lambda})=h^{-1}g\cdot f(x_{\lambda}). $$ 
Since $ G $  is finite, $ G_{x_{\lambda}}= G_{f(x_{\lambda})} $ , then
$$  h^{-1}g\cdot f(x_{\lambda})= f(x_{\lambda}) \Rightarrow g\cdot f(x_{\lambda})=h\cdot f(x_{\lambda}) \Rightarrow \widetilde{f}(g\cdot x_{\lambda})=\widetilde{f}(h\cdot x_{\lambda}). $$ 

We must verify that $ \widetilde{f} $  is also $ G $ -equivariant. Let $ z\in \ibox $ , then there exist $ g\in G $ , $ \lambda \in \Lambda $ , and $ x_{\lambda}\in X_{i} $  such that $ z= g\cdot x_{\lambda} $ . It follows that
$$  \widetilde{f}(h\cdot z)=\widetilde{f}(hg\cdot x_{\lambda})= hg\cdot f(x_{\lambda})=h\cdot (g\cdot f(x_{\lambda}))=h\cdot \widetilde{f}(g\cdot x_{\lambda})=h\cdot \widetilde{f}(z). $$ 

Note that $ f $  is the restriction of $ \widetilde{f} $  to the set of representatives of the $ G $ -orbits,
$$  f=\widetilde{f}|_{X_{i}}. $$ 

Another important observation is that a function $ \tau \in \EndH $  may not be well-defined when restricted to the set of representatives of $ G $ -orbits. This happens because if $ x\in X_{i}\subseteq\ibox $ , it is possible that $ \tau(x)\notin X_{i} $ .

\begin{lema}
    Let $ G $  be a finite group, $ H_{i} $  a subgroup, and $ X_{i} $  the set of representatives of $ G $ -orbits in $ \ibox $ , then:
    \begin{enumerate}
        \item[i)] For $ \tau \in \EndH $  such that $ \tau(x)\in X_{i} $  for all $ x \in X_{i} $ , we have $ \widetilde{(\tau|_{X_{i}})}=\tau $ .
        \item [ii)] For $ \sigma \in \EndG $  such that $ \sigma(x)\in \ibox $  for all $ x\in \ibox $ , we have $ \widehat{(\sigma|_{\ibox})}=\sigma $ .
    \end{enumerate}
\end{lema}

Given a set $ S $  from the monoid $ Trans(X_{i}) $ , denote $ \tilde{S}:=\{\tilde{f}|\ f\in S\} $ . Then for every $ f\in \langle S \rangle $ , we have $ \tilde{f} \in \langle \tilde{S} \rangle $ . Similarly, for a subset $ M $  of $ \EndH $ , denote $ \widehat{M}:=\{\widehat{\tau}|\ \tau \in M\} $ . Then for every $ \tau \in \langle M \rangle $ , $ \widehat{\tau} \in \langle \widehat{M} \rangle $ .

\noindent When a finite group $ G $  acts on a set $ X $ , the following result is well known.

\begin{teorema}
    Let $ G $  be a finite group acting on a set $ X $ , and let $ \tau \in \EndG $ . Then it holds that:
    $$  \tau=\tau_{1}\tau_{2} ... \tau_{r}, $$ 
    where each $ \tau_{i}\in \EndG $  is defined as:
    $$  \tau_{i}(z)= \left\{  \begin{array}{cc}
        \tau(z) &  z\in \ibox, \\ 
        z & \mbox{otherwise.}
    \end{array} \right. $$ 
\end{teorema}

\begin{proof}
    Suppose $ z\in\mathcal{B}_{[H_{i}]} $ , then $ \tau_{k}(z)=z $  for all $ k>i $ . By definition, $ \tau_{i}(z)=\tau(z)\in\mathcal{B}_{[H_{j}]} $ , where $ i \leq j $ , by Lemma \ref{lema1}. Thus $ \tau_{k}(\tau_{i}(z))=\tau_{i}(z)=\tau(z) $  for all $ k<i $ , and consequently $ \tau_{1}  \tau_{2}  \cdots  \tau_{r}(z)=\tau(z) $ .
\end{proof}

\section{Case 1: Group as stabilizer}

Consider $ G $  a finite group and $ X $  an infinite $ G $ -set such that $ |\ibox| $  is finite for each subgroup $ H_{i}\leq G $  except the group itself, $ H_{r}=G $ .

Since $ G $  is a finite group, the orbit-stabilizer theorem tells us that each $ G $ -orbit in $ \Gbox $  is a singleton set. Therefore, according to \cite[Proposition 9]{paper}, we ensure that
$$  \EndGbox \cong Trans(\mathcal{B}_{[G]}). $$ 

Given that $ \Gbox $  is an infinite set, by \cite[Theorem 3.3]{cite3}, there exist two functions $ \mu,\nu \in Trans(\Gbox) $  such that
$$  Trans(\Gbox) = \langle Sym(\Gbox), \mu,\nu \rangle, $$ 
thus, we assert that there exist two $ G $ -equivariant functions $ \widetilde{\mu},\widetilde{\nu} \in \EndGbox $  such that
$$  \EndGbox = \langle \AutGbox,\widetilde{\mu},\widetilde{\nu}\rangle. $$ 

For simplicity in notation, if we have a set of representatives of $ G $ -orbits for some $ H_{i} $ , denoted by $ X_{i} $ , and if we have a function $ f: X_{i} \mapsto X_{i} $ , we will simply denote the extension of $ f $  to all of $ X $  as $ \widehat{f}=\widehat{\tilde{f}} $ .

Given the action of $ G $  on $ X $ , we can select without loss of generality the subgroup representatives of the conjugacy classes included in $ \ConjX $  such that there are always containments between these, i.e., $ H_{i}\leq H_{j} $ .

Also, for each $ i=1,2,...,r $ , we can select a representative element from each $ \ibox $ , $ x_{i}\in \ibox $ , such that $ G_{x_{i}}=H_{i} $ , and this is contained in the set of representatives of $ \ibox $ , $ X_{i} $ .

For each $ i=1,2,...,r $  where $ \ibox $  has more than one $ G $ -orbit, we can fix a second element $ x'_{i}\in \ibox $  such that $ G_{x_{i}}=G_{x'_{i}}=H_{i} $ , and it is also in the set of representatives of orbits of $ \ibox $  and furthermore $ G_{x_{i}}\neq G_{x'_{i}} $ .

Furthermore, for each $ i=1,2,...,r $ , consider the different $ N_{i} $ -conjugacy classes $ [K_{i,1}]_{N_{i}},[K_{i,2}]_{N_{i}},...,[K_{i,r_{i}}]_{N_{i}} $ such that $ H_{i}< K_{i,t} $  and $ K_{i,t}\in \StabG $ . Then for each $ j=1,...,r_{i} $ , we can fix elements $ y_{i,j}\in X $  such that $ G_{y_{i,j}}=K_{i,j} $ .

From these fixed elements, we define the following $ G $ -equivariant functions:
$$  [x\mapsto y](z)= \left\{  \begin{array}{cc}
    g\cdot y &  z=g\cdot x, \\
    z & \mbox{otherwise.}
\end{array} \right. $$ 

Note that $ [x_{i} \mapsto y_{i,j}] $  is an elementary collapse of type $ (H_{i},[K_{i,j}]_{N_{i}}) $  and $ [x_{i}\mapsto x'_{i}] $  is an elementary collapse of type $ (H_{i},[H_{i}]_{N_{i}}) $ .

We also define bijective $ G $ -equivariant functions as follows. For $ x,y\in X $  such that $ G_{x}=G_{y} $ , we define:
$$  (x\leftrightarrow y)(z)= \left\{  \begin{array}{cc}
    g\cdot y &  z=g\cdot x, \\
    g\cdot x &  z=g\cdot y, \\
    z & \mbox{otherwise.}
\end{array} \right. $$ 

An important observation is that each of these bijective functions satisfies being its own inverse, i.e., if $ (x \leftrightarrow y) $  exists, then $ (x\leftrightarrow y)^{2}=id_{X} $ .

Denote by $ W $  the set of all elementary collapses formed with fixed elements in $ X $  under the action of $ G $ .

\begin{teorema}\label{version1}
Let $ G $  be a finite group acting on a set $ X $  such that $ |\ibox| $  is finite for each subgroup $ H_{i}\leq G $  except the whole group, $ H_{r}=G $ . Let $ W $  be the set of elementary collapses of fixed elements by the action. Then it holds that
$$  \EndG \neq \langle \AutG \cup W \rangle. $$ 
\end{teorema}

\begin{proof}
We will proceed with a proof by contradiction. Suppose $ \AutG \cup W $  does generate the entire monoid. Note that elementary collapses are functions that are neither injective nor surjective; in particular, it holds that $ def([x\mapsto y])>0 $  for any $ x,y\in X $ , since it is not surjective. Let $ \tau \in \EndG $  be a surjective but not injective $ G $ -equivariant function, i.e., $ \tau \in \EndG \setminus \AutG $ . Then $ def(\tau)=0 $ , and it can be expressed as
$$  \tau= f_{1}f_{2}...f_{k}, $$ 
where $ f_{i} \in \AutG \cup W $ . Therefore, it must hold that
$$  def(\tau) \geq \max \{def(f_{1}), def(f_{2}), ..., def(f_{k})\}, $$ 
hence $ def(f_{1})=def(f_{2})=...=def(f_{k})=0 $ . Since no elementary collapse in $ W $  has defect zero, it follows that $ f_{i} \in \AutG $  for all $ i=1,2,...,k $ . This implies that $ \tau = f_{1}f_{2}...f_{k} \in \AutG $ , which contradicts the choice of $ \tau $ .
\end{proof}

Consider the functions $ \mu $  and $ \nu $  which generate $ Trans(X_{r}) $ . Let $ V $  be the set obtained by removing the elementary collapse of type $ (H_{r},[H_{r}]_{N_{r}}) $  from $ W $  and adding the extensions of the functions $ \mu $  and $ \nu $ . That is,
$$  V= \left( W \setminus \{[x_{r} \mapsto x'_{r}]\}\right) \cup \{\widehat{\mu},\widehat{\nu}\}. $$ 

\begin{teorema}
Let $ G $  be a finite group acting on an infinite set $ X $  such that $ |\ibox| $  is finite for each subgroup $ H_{i}\leq G $  except the whole group $ G $ . Then it holds that
$$  \EndG = \langle \AutG \cup V \rangle. $$ 
\end{teorema}

\begin{proof}
Given $ \tau \in \EndG $ , since it can be expressed as
$$  \tau=\tau_{1}\tau_{2}...\tau_{r}, $$ 
it suffices to prove that $ \tau_{i} \in \AutG \cup V $  for all $ i=1,2,...,r $ .

Consider the sets
$$  \mathcal{B}'_{i}:=\{z\in \ibox \mid \tau_{i}(z)\in \ibox\}, $$ 
$$  \mathcal{B}''_{i}:=\{z\in \ibox \mid \tau_{i}(z)\notin \ibox\}. $$ 
Since for each $ i \neq r $ , $ \ibox $  is a finite set, $ \mathcal{B}'_{i} $  and $ \mathcal{B}''_{i} $  are finite and $ G $ -invariant, i.e., composed of a finite number of $ G $ -orbits each. Based on these, we define functions $ \tau'_{i},\tau''_{i}: X \rightarrow X $ , given by
$$  \tau'_{i}(z)= \left\{  \begin{array}{cc}
\tau_{i}(z) &  z\in \mathcal{B}'_{i}, \\ z & \mbox{otherwise.}\end{array} \right., $$ 
$$  \tau''_{i}(z)= \left\{  \begin{array}{cc}
\tau_{i}(z) &  z\in \mathcal{B}''_{i}, \\ z & \mbox{otherwise.}\end{array} \right. $$ 
It is straightforward to verify that $ \tau_{i}=\tau'_{i}\tau''_{i} $ .

To complete the proof of this theorem, it suffices to show that the functions $ \tau'_{i} $  and $ \tau''_{i} $  belong to $ \langle \AutG \cup V \rangle $ .

A particular case in this construction is for $ i=r $ , where $ \mathcal{B}''_{r}= \emptyset $  and $ \tau''_{r}= id_{X} $ . Note that for $ i \neq r $ , as $ \ibox $  is a finite set, $ \mathcal{B}''_{i} $  is a finite and $ G $ -invariant set, i.e., it can be viewed as a union of $ G $ -orbits, specifically a finite number of these. Let's say $ \mathcal{B}''_{i}= \mathcal{O}_{1}\cup\mathcal{O}_{2}\cup...\cup\mathcal{O}_{m_{i}} $ . We can select $ m_{i} $  representatives of these orbits, say $ x_{1}^{i},x_{2}^{i},...,x_{m_{i}}^{i}\in \mathcal{B}''_{i} $ , such that $ G_{x_{k}^{i}}=H_{i} $ , $ k=1,2,...,m_{i} $ . For each $ x_{k}^{i} $ , there exists an index $ j=1,2,...,r_{i} $  such that $ [G_{\tau(x_{k}^{i})}]_{N_{i}}=[K_{i,j}]_{N_{i}} $ . Notice that $ j $  depends both on $ k $  and $ i $ , and there exists $ n\in N_{i} $  such that $ G_{\tau(x_{k}^{i})}=n^{-1}K_{i,j}n =n^{-1} G_{y_{i,j}}n $ . Therefore, it holds that:

\begin{scriptsize}
$$  \tau''_{i}= \prod_{k=1}^{m_{i}}  (x_{i} \leftrightarrow x^{i}_{k})(x_{i} \leftrightarrow n\cdot x_{i})(y_{i,j} \leftrightarrow n^{-1}\cdot \tau(x_{k}^{i}))[x_{i} \mapsto y_{i,j}](y_{i,j} \leftrightarrow n^{-1}\cdot \tau(x_{k}^{i}))(x_{i} \leftrightarrow n\cdot x_{i})(x_{i} \leftrightarrow x^{i}_{k}). $$ 
\end{scriptsize}

As this is a finite product, it shows that $ \tau''_{i}\in \langle \AutG \cup V \rangle $ .

Now, for $ \tau'_{i} $ , consider first that for $ i \neq r $ , the restriction of $ \tau'_{i} $  to $ \ibox $  is contained in a submonoid of $ \EndG $  isomorphic to $ \EndHi $ . If $ \ibox $  has only one $ G $ -orbit, then $ \EndHi = \EndGx = \AutGx $  for some $ x \in \ibox $ , so $ (\tau'_{i}|_{\ibox})\in \AutGx $ , and its extension is an automorphism of $ X $ , i.e., $ \widehat{(\tau'_{i}|_{\ibox})}=\tau'_{i} \in \AutG $ . Then, if $ \ibox $  has more than one $ G $ -orbit, according to Proposition \cite[Proposition 9]{paper}, $ \EndHi \cong \AutGx \wr Trans(\mathcal{O}_{[H_i]}) $ , for any $ x \in \ibox $ . As $ \ibox $  is a finite set, the monoid $ Trans( \mathcal{O}_{[H_i]}) $  is generated by $ Sym(\mathcal{O}_{[H_i]}) $  along with any function of defect 1 (see \cite[Prop. 1.2]{cite3}). For any $ x_{i}, x'_{i} \in \ibox $ , since $ Gx_{i}\neq Gx'_{i} $ , the function $ [x_{i} \mapsto x'_{i}] $  induces, with its restriction, a function over $ \mathcal{O}_{H_i} $  of defect 1. It follows that $ \EndHi $  is generated by $ \AutHi \cup \{[ x_{i} \mapsto x'_{i} ] \vert_{\mathcal{B}_i} \} $ , thus $ (\tau'_{i}|_{\ibox})\in \langle \AutHi \cup \{[ x_{i} \mapsto x'_{i} ] \vert_{\mathcal{B}_i}\rangle $ . Therefore, $ \widehat{(\tau'_{i}|_{\ibox})}=\tau'_{i}\in \langle\AutG \cup \{[x'_{i} \mapsto x'_{i}]\} \rangle \subseteq \langle \AutG \cup V\rangle $ .

For the case where $ i=r $ , $ \rbox $  is infinite, so the monoid $ Trans( \mathcal{O}_{[H_r]}) $  is isomorphic to the monoid of transformations of representatives of $ G $ -orbits in $ \rbox $ , and thus is generated by $ Sym(\mathcal{O}_{[H_r]}) $  along with the functions $ \mu $  and $ \nu $ . This implies that $ \EndHr $  is generated by $ \AutHr $  along with the inductions of the functions $ \mu $  and $ \nu $ , i.e., $ (\tau'_{r}|\rbox)\in \EndHr = \langle\AutHr, \tilde{\mu},\tilde{\nu} \rangle $ , and consequently $ \widehat{(\tau'_{r}|_{\rbox})}= \tau'_{r}\in \langle \AutG \cup \{\widehat{\mu},\widehat{\nu}\}\rangle\subseteq \langle \AutG \cup V \rangle $ .
\end{proof}

\section{Case 2: The Trivial Subgroup}

Consider the same finite group $ G $ , but now $ X $  is an infinite $ G $ -set such that $ |\Hbox| $  is finite for each subgroup $ H \leq G $  except the trivial subgroup $ H_{1} = \{e\} $ .

\begin{teorema}
Let $ G $  be a finite group acting on a set $ X $  such that $ |\ibox| $  is finite for each subgroup $ H_{i} \leq G $  except the trivial subgroup $ H_{1} = \{e\} $ , and let $ W $  be the set of elementary collapses of fixed elements by the action. Then it holds that 
$$  \EndG \neq \langle \AutG \cup W \rangle. $$ 
\end{teorema}

\begin{proof}
The argument is exactly the same as for Theorem \ref{version1}.
\end{proof}

Now consider the functions $ \mu $  and $ \nu $  as those generating $ Trans(X_{1}) $ . Also, consider the set $ V $  resulting from removing the elementary collapse of type $ (H_{1}, [H_{1}]_{N_{1}}) $  from $ W $  and adding the extensions of the functions $ \mu $  and $ \nu $ . That is,
$$  V = \left( W \setminus \{[x_{1} \mapsto x'_{1}]\}\right) \cup \{\widehat{\mu}, \widehat{\nu}\}. $$ 

\begin{teorema}
Let $ G $  be a finite group acting on an infinite set $ X $  such that $ |\ibox| $  is finite for each subgroup $ H_{i} \leq G $  except the trivial subgroup $ H_{1} = \{e\} $ . Then it holds that 
$$  \EndG = \langle \AutG \cup V \rangle. $$ 
\end{teorema}

\begin{proof}
Given an element $ \tau \in \EndG $  and its decomposition $ \tau = \tau_{1}\tau_{2}...\tau_{r} $ , consider again the sets
$$  \mathcal{B}'_{i} := \{z \in \ibox \mid \tau_{i}(z) \in \ibox\}, $$ 
$$  \mathcal{B}''_{i} := \{z \in \ibox \mid \tau_{i}(z) \notin \ibox\}, $$ 
and the functions
$$  \tau'_{i}(z) = \begin{cases}
\tau_{i}(z) & \text{if } z \in \mathcal{B}'_{i}, \\
z & \text{otherwise.}
\end{cases} $$ 
$$  \tau''_{i}(z) = \begin{cases}
\tau_{i}(z) & \text{if } z \in \mathcal{B}''_{i}, \\
z & \text{otherwise.}
\end{cases} $$ 
Thus, $ \tau_{i} = \tau'_{i}\tau''_{i} $ , and we need to show $ \tau'_{i}, \tau''_{i} \in \langle \AutG \cup V \rangle $ .

By arguments similar to those presented in the previous case, it follows that
$$  \tau'_{i} \in \langle \AutG, [x_{i} \mapsto x'_{i}] \rangle, \text{ for } i \neq 1, $$ 
$$  \tau'_{1} \in \langle \AutG, \widehat{\mu}, \widehat{\nu} \rangle, $$ 
and for $ i \neq 1 $ ,
\begin{scriptsize}
$$  \tau''_{i} = \prod_{k=1}^{m_{i}} (x_{i} \leftrightarrow x^{i}_{k})(x_{i} \leftrightarrow n\cdot x_{i})(y_{i,j} \leftrightarrow n^{-1}\cdot \tau(x_{k}^{i}))[x_{i} \mapsto y_{i,j}](y_{i,j} \leftrightarrow n^{-1}\cdot \tau(x_{k}^{i}))(x_{i} \leftrightarrow n\cdot x_{i})(x_{i} \leftrightarrow x^{i}_{k}), $$ 
\end{scriptsize}
which belongs to $\langle \AutG, [x_{i} \mapsto y_{i,j}] \rangle$.\\

The remaining function $ \tau''_{1} $  has a slightly different construction. Notice that by hypothesis, $ X \setminus \unobox $  is finite, so we can enumerate its elements as $ X \setminus \unobox = \{\overline{z}_{1}, \overline{z}_{2}, ..., \overline{z}_{k'}\} $ , which is $ G $ -invariant and can be viewed as a finite union of $ G $ -orbits $ \mathcal{O}_{1}, \mathcal{O}_{2}, ..., \mathcal{O}_{k} $ .

Therefore, we can generate a finite partition of $ \mathcal{B}''_{1} $  given by
$$  \mathcal{B}^{1}_{t} := \{x \in \mathcal{B}''_{1} \mid \tau(x) \in \mathcal{O}_{t}\}, $$ 
so that
$$  \mathcal{B}''_{1} = \bigcup_{t=1}^{k} \mathcal{B}^{1}_{t} = \mathcal{B}^{1}_{1} \cup \mathcal{B}^{1}_{2} \cup ... \cup \mathcal{B}^{1}_{k}. $$ 

Since each $ G $ -orbit is finite, it can be viewed as
$$  \mathcal{O}_{k} = \{ z^{k}_{1}, z^{k}_{2}, ..., z^{k}_{m_{k}} \}. $$ 
Given the set of representatives of $ G $ -orbits in $ \unobox $ , denoted by $ X_{1} $ , we consider the set
$$  X^{1}_{k,p} = \{ x \in X_{1} \mid \tau(x) = z^{k}_{p} \}, $$ 
if $X^{1}_{k,p}$ is not empty, there exists at least one element $a^{k}_{p}\in X^{1}_{k,p}$, and there exist functions $ f_{k,p}: X_{1} \rightarrow X_{1} $ which will be selected in function of the following cases.\\
If $x_{1}\in X^{1}_{k,p}$, we define $f_{k,p}$ as follows:
$$  f_{k,p}(z) = \begin{cases}
x_{1} & \text{if } z \in X^{1}_{k,p} \\
z & \text{otherwise.}
\end{cases} $$ 
If $x_{1}\notin X^{1}_{k,p}$, we define $f_{k,p}$ as:
$$  f_{k,p}(z) = \begin{cases}
x_{1} & \text{if } z \in X^{1}_{k,p} \\
a^{k}_{p} & \text{if } z=x_{1} \\
z & \text{otherwise.}
\end{cases} $$ 
If $X^{1}_{k,p}$ is empty, we consider $f_{k,p}$ as the identity function.\\

Note that $ f_{k,p} \in \langle Sym(X_{1}), \mu, \nu \rangle $, and consequently $ \widehat{f_{k,p}} \in \langle \AutG, \widehat{\mu}, \widehat{\nu} \rangle $. Also note that for each $ z^{k}_{q} $ , there exists a $ y_{i,j} $  such that $ [G_{z^{k}_{q}}]_{N_{1}} = [G_{y_{i,j}}]_{N_{1}} $. This implies there exists an element $ n \in N_{1} $  such that $ G_{z^{k}_{q}} = n G_{y_{i,j}} n^{-1} $. It is also mentioned that the indices $ j $  depends on both $ q $ and $ k $.\\

To simplify notation we built the following functions, based on the same cases that $f_{k,p}$:
If $x_{1}\in X^{1}_{k,p}$
$$\tau_{t,p}= (y_{i,j} \leftrightarrow n^{-1} \cdot z^{t}_{p})[x_{1} \mapsto y_{i,j(t,p)}](y_{i,j} \leftrightarrow n^{-1} \cdot z^{t}_{p})(x_{1} \leftrightarrow n \cdot x_{1}) \widehat{f_{t,p}},$$If $x_{1}\notin X^{1}_{k,p}$
$$\tau_{t,p}= (x_{1} \leftrightarrow a^{k}_{p})(y_{i,j} \leftrightarrow n^{-1} \cdot z^{t}_{p})[x_{1} \mapsto y_{i,j(t,p)}](y_{i,j} \leftrightarrow n^{-1} \cdot z^{t}_{p})(x_{1} \leftrightarrow n \cdot x_{1}) \widehat{f_{t,p}},$$

If $X^{1}_{k,p}$ is empty or such $ y_{i,j} $  does not exist for some $ z^{k}_{q} $ , then $ \tau_{k,q} $  would be considered as the identity function. \\

Therefore, it follows that
$$\tau''_{1}= \prod_{t=1}^{k} \prod_{p=1}^{m_{t}}  \tau_{t,p},$$
which proves the main statement of  this theorem.

\end{proof}

\begin{corolario}
Let $G$ be a finite group acting over a set $X$ in any of the following cases:

\begin{enumerate}
\item[i)] $|\Hbox| < \infty$, $\forall H \in \StabG \setminus \{\langle e \rangle\}$, and $|\mathcal{B}_{[\{e\}]}|= \infty$. 
\item[ii)] $|\Hbox| < \infty$, $\forall H \in \StabG \setminus \{G\}$, and $|\mathcal{B}_{[G]}|= \infty$.

\end{enumerate}
Then it holds that:
$$rank(\EndG : \AutG)=\sum_{[H]\in \ConjX}{|U(H)|}-|\kappa_{G}(X)|+1.$$
\end{corolario}

\end{document}